\documentclass[a4paper,draft]{amsproc}
\usepackage{amssymb}
\usepackage{amscd} 
\theoremstyle{plain}
\newtheorem*{theoA}{Theorem A}
\newtheorem*{theoB}{Theorem B}
\newtheorem*{theoC}{Theorem C}
\newtheorem*{theoD}{Theorem D}
\newtheorem*{theoE}{Theorem E}
\newtheorem*{theoF}{Theorem F}
\newtheorem*{theoG}{Theorem G}
\newtheorem*{theoH}{Theorem H}

 \newtheorem{theo}{Theorem}[section]
 
 \newtheorem{lem}{Lemma}[section]
 \newtheorem{cor}{Corollary}[section]
\theoremstyle{definition}
 \newtheorem{exm}{Example}[section]
 \newtheorem{ques}{Question}[section]
 \newtheorem{defi}{Definition}[section]
 \newtheorem{note}{Note}[section]
\theoremstyle{remark}
 \newtheorem{rem}{Remark}[section]
 
 \newcommand{\ol}{\overline}
\newcommand{\be}{\begin{equation}}
\newcommand{\ee}{\end{equation}}
\newcommand{\beas}{\begin{eqnarray*}}
\newcommand{\eeas}{\end{eqnarray*}}
\newcommand{\bea}{\begin{eqnarray}}
\newcommand{\eea}{\end{eqnarray}}
\newcommand{\lra}{\longrightarrow}
 \numberwithin{equation}{section}

\renewcommand{\leq}{\leqslant}
\renewcommand{\geq}{\geqslant}
\renewcommand{\setminus}{\smallsetminus}
\title[On uniqueness of two meromorphic functions ...]{\LARGE On uniqueness of two meromorphic functions sharing three sets with finite weights}

\numberwithin {equation}{section}
\date{}
\author{Molla Basir Ahamed}
\address{ Department of Mathematics, Kalipada Ghosh Tarai Mahavidyalya, West Bengal, 734014, India.}
\email{basir\_math\_kgtm@yahoo.com, bsrhmd117@gmail.com.}
\begin{document}
\vspace{18mm} \setcounter{page}{1} \thispagestyle{empty}

\begin{abstract}
With the help of the notion of weighted sharing of sets, this paper dealt with the question posed by \emph{Yi} \cite{Yi-SC-1994} regarding the uniqueness of meromorphic functions concerning three set sharing. A result has been proved which significantly improved the recent results of \emph{Banerjee - Ahamed} \cite{Ban & Aha-BPAS-2014}, \emph{Banerjee - Mukherjee} \cite{Ban & Muk-HJ-2008} and \emph{Banerjee - Majumder} \cite{Ban & Maj-A-2014} by relaxing the nature of sharing. Several examples have been exhibited to show the sharpness of the cardinalities of the sets $\mathcal{S}_1$ and $\mathcal{S}_2$ considered in \emph{Theorem \ref{t1.2}}\;. Moreover, we give some constructive examples to endorse the validity of our established theorem.
\end{abstract}
\maketitle

\section{\sc Introduction, Definitions and Results}
In this paper by a meromorphic function we will always mean a meromorphic function in the open complex plane. Let $f$ and $g$ be two non-constant meromorphic functions and let $a\in\mathbb{C}\cup\{\infty\}$. For standard definitions and notations of value distribution theory we refer to the reader to see \cite{Hay-1964}. We denote through out the paper $\mathbb{C^{*}}=\mathbb{C}\setminus\{0\}$.\par 
If $f$ and $g$ have the same set of $a$-points with same multiplicities then we say that $f$ and $g$ share the value $a$ $CM$ (Counting Multiplicities). If we do not take the multiplicities into account, $f$ and $g$ are said to share the value $a$ $IM$ (Ignoring Multiplicities).\par 

\begin{defi}
	For a non-constant meromorphic function $f$ and any set $\mathcal{S}\subset\mathbb{\ol C}$, we define \beas E_{f}(\mathcal{S})=\displaystyle\bigcup_{a\in\mathcal{S}}\bigg\{(z,p)\in\mathbb{C}\times\mathbb{N}:f(z)=a,\;\text{with multiplicity}\; p\bigg\}, \eeas \beas\ol E_{f}(\mathcal{S})=\displaystyle\bigcup_{a\in\mathcal{S}}\bigg\{(z,1)\in\mathbb{C}\times\{1\}:f(z)=a\bigg\}.\eeas
\end{defi}
\par If $E_{f}(\mathcal{S})=E_{g}(\mathcal{S})$ ($\ol E_{f}(\mathcal{S})=\ol E_{g}(\mathcal{S})$) then we simply say $f$ and $g$ share $\mathcal{S}$ Counting Multiplicities(CM) (Ignoring Multiplicities(IM)).\par
Evidently, if $\mathcal{S}$ contains one element only, then it coincides with the usual definition of $CM (IM)$ sharing of values. \par 

Next we explain some definitions and notations which will be used in the paper.
\begin{defi}\cite{Lah-Sar-2004} Let $p$ be a positive integer and $a\in\mathbb{C}\cup\{\infty\}$.\begin{enumerate}
		\item[(i)] $N\left(r,\displaystyle\frac{1}{f-a}\mid \geq p\right)$ $\left(\ol N\left(r,\displaystyle\frac{1}{f-a}\mid \geq p\right)\right)$ denotes the counting function (reduced counting function) of those $a$-points of $f$ whose multiplicities are not less than $p$.
		\item[(ii)] $N\left(r,\displaystyle\frac{1}{f-a}\mid \leq p\right)$ $\left(\ol N\left(r,\displaystyle\frac{1}{f-a}\mid \leq p\right)\right)$ denotes the counting function (reduced counting function) of those $a$-points of $f$ whose multiplicities are not greater than $p$.
	\end{enumerate} 
\end{defi}
\begin{defi}
	Let $f$ and $g$ be two non-constant meromorphic functions such that $f$ and $g$ share the value $a$ with weight $k$ where $a\in\mathbb{C}\cup\{\infty\}$. Let $ f $ and $ g $ have same $ a $-points with respective multiplicities $ p $ and $ q $. We denote by $\ol N_E^{(k+1}\left(r,\displaystyle\frac{1}{f-a}\right)$ the counting function of those $a$-points of $f$ and $g$ where $p=q\geq k+1$, each point in this counting function counted only once.  
\end{defi}
\begin{defi}\cite {Yi-1991} For $a\in\mathbb{C}\cup\{\infty\}$ and a positive integer $p$ we denote by \beas N_{p}\left(r,\frac{1}{f-a}\right)=\ol N\left(r,\frac{1}{f-a}\right)+\ol  N\left(r,\frac{1}{f-a}\mid\geq 2\right)+\ldots+\ol N\left(r,\frac{1}{f-a}\mid\geq p\right).\eeas\par It is clear that $N_{1}\left(r,\displaystyle\frac{1}{f-a}\right)=\ol N\left(r,\displaystyle\frac{1}{f-a}\right)$.
\end{defi}
\begin{defi}
	Let $N_{1)}\left(r,\displaystyle\frac{1}{f-a}\right)$ denote the counting function of the simple zeros of $f-a$ and $\ol N_{(2}\left(r,\displaystyle\frac{1}{f-a}\right)$ denote the reduced counting function of the $a$-points of $f$ of multiplicities $\geq 2$. It follows that \beas N_2\left(r,\frac{1}{f-a}\right)=N_{1)}\left(r,\displaystyle\frac{1}{f-a}\right)+2\ol N_{(2}\left(r,\displaystyle\frac{1}{f-a}\right).\eeas
\end{defi}
\begin{defi}\cite{Zha-2005} For a positive integer $p$ and $a\in\mathbb{C}\cup\{\infty\}$, we put \beas\delta_{p}(a;f)= 1- \limsup\limits _{r\lra \infty}\frac{N_{p}\left(r,\displaystyle\frac{1}{f-a}\right)}{T(r,f)}.\eeas  \beas\Theta(a;f)= 1- \limsup\limits _{r\lra \infty}\frac{\ol N\left(r,\displaystyle\frac{1}{f-a}\right)}{T(r,f)}\eeas 
	Clearly $0\leq \delta (a;f)\leq \delta _{p}(a;f)\leq \delta_{p-1}(a;f)\leq\ldots \leq\delta_{2}(a;f)\leq\delta_{1}(a;f)=\Theta (a;f)$. 
\end{defi}\par 
In $1926$, \emph{Nevanlinna} first showed that a non-constant meromorphic function on the complex plane $\mathbb{C}$ is uniquely determined by the pre-images, ignoring multiplicities, of $5$ distinct values (including infinity). A few years latter, he showed that when multiplicities are taken into consideration, $4$ points are enough and in that case either the two functions coincides or one is the bilinear transformation of the other
one.\par The uniqueness problem for entire or meromorphic functions sharing sets was initiated by a famous question of \emph{Gross} in \cite{Gross}. 
In $1976$, \emph{Gross} \cite{Gross} asked the following question.
\begin{ques}\label{qn1.1}
	Can one find two finite sets $\mathcal{S}_{j}, (j=1,2)$ such that any two non-constant entire functions $f$ and $g$ satisfying $E_{f}(\mathcal{S})=E_{g}(\mathcal{S})$, $(j=1,2)$ must be identical ?
\end{ques}\par In \cite{Gross}, \emph{Gross} said that if the answer of \emph{\sc Question \ref{qn1.1}} is affirmative it would be interesting to know how large both sets would have to be ?\par In $1994$, \emph{Yi} \cite{Yi-SC-1994} posed the following question.
\begin{ques}\label{qn1.2}
	Can one find three finite sets $\mathcal{S}_{j}, (j=1,2, 3)$ such that any two non-constant meromorphic functions $f$ and $g$ satisfying $E_{f}(\mathcal{S})=E_{g}(\mathcal{S})$, $(j=1,2,3)$ must be identical ?
\end{ques} \par In the same paper \cite{Yi-SC-1994}, \emph{Yi} answered the \emph{\sc Question \ref{qn1.2}} affirmatively and obtained a result by showing that there exist three finite sets $\mathcal{S}_1$ (with $7$ elements), $\mathcal{S}_2$ (with $2$ elements) and $\mathcal{S}_3$ (with $1$ element) such that any two non-constant meromorphic functions $f$ and $g$ satisfying $E_{f}(\mathcal{S}_{j})=E_{g}(\mathcal{S}_{j})$, $(j=1,2,3)$ must be identical.\par 
In the direction of \emph{\sc Question \ref{qn1.2}}, \emph{Fang - Xu} \cite{Fan & Xu-CJCM-1997} obtained the following result.
\begin{theoA}\cite{Fan & Xu-CJCM-1997}
	Let $\mathcal{S}_1=\{0\}$, $\mathcal{S}_2=\{z:z^3-z^2-1=0\}$ and $\mathcal{S}_3=\{\infty\}$. Let $f$ and $g$ be two non-constant meromorphic functions such that $\Theta(\infty;f)>\displaystyle\frac{1}{2}$ and $\Theta(\infty;g)>\displaystyle\frac{1}{2}$. If $E_{f}(\mathcal{S}_{j})=E_{g}(\mathcal{S}_{j})$, for $j=1,2,3$ then $f\equiv g$.    
\end{theoA}\par Dealing with the \emph{ Question \ref{qn1.2}}, \emph{Qiu - Fang} \cite{Qiu & Fan-BSMS-2002} obtained a result with an extra supposition that the meromorphic functions $f$ and $g$ both having poles of multiplicity $\geq 2$. In the same paper they also exhibited some examples to show that the condition on the poles of $f$ and $g$ can not be removed.  
\par In $2004$, \emph{Yi - Lin} \cite{Yi & Lin-PJAS-2004} proved the following results.
\begin{theoB}\cite{Yi & Lin-PJAS-2004}
	Let $\mathcal{S}_1=\{0\}$, $\mathcal{S}_2=\{z:z^n+bz^{n-1}+c=0\}$ and $\mathcal{S}_3=\{\infty\}$, where $a$, $b$ are non-zero constants such that $z^n+bz^{n-1}+c=0$ has no repeated root and $n\geq 3$ is an integer. If for two non-constant meromorphic functions $f$ and $g$, $E_{f}(\mathcal{S}_{j})=E_{g}(\mathcal{S}_{j})$, for $j=1,2,3$ and $\delta_{1}(\infty;f)>\displaystyle\frac{5}{6}$, then $f\equiv g$.   
\end{theoB}
\begin{theoC}\cite{Yi & Lin-PJAS-2004}
	Let $\mathcal{S}_1=\{0\}$, $\mathcal{S}_2=\{z:z^n+bz^{n-1}+c=0\}$ and $\mathcal{S}_3=\{\infty\}$, where $a$, $b$ are non-zero constants such that $z^n+bz^{n-1}+c=0$ has no repeated root and $n\geq 4$ is an integer. If for two non-constant meromorphic functions $f$ and $g$, $E_{f}(\mathcal{S}_{j})=E_{g}(\mathcal{S}_{j})$, for $j=1,2,3$ and $\Theta(\infty;f)>0$, then $f\equiv g$.   
\end{theoC}	\par 
Progressively the research on \emph{\sc Question \ref{qn1.1}} for meromorphic functions as well as \emph{\sc Question \ref{qn1.2}} gained a valuable space in the literature and now-a-days it has increasingly become an impressive  branch of the modern uniqueness theory of meromorphic functions. During the last few years a considerable amount of work has been done to explore the possible answer to \emph{\sc Question \ref{qn1.2}} by many Mathematicians. \par In $2001$, the introduction of the new notion of sharing which is a scaling between $CM$ or $IM$, known as weighted sharing of values and sets by \emph{Lahiri} \emph{\cite{Lah-NMJ-2001, Lah-CVTA-2001}} further speed up the research in the direction of \emph{Question \ref{qn1.2}}.
\begin{defi}
	Let $k\in\mathbb{N}\cup\{0\}\cup\{\infty\}$. For $a\in\mathbb{C}\cup\{\infty\}$, we denote by $E_{f}(a,k)$ the set of all $a$-points of $f$, where an $a$-point of multiplicity $m$ is counted $m$ times if $m\leq k$ and $k+1$ times if $m\geq k+1$. If $E_{f}(a,k)=E_{g}(a,k)$, we say that $f$ and $g$ share the value $a$ with weight $k$.  
\end{defi} 
\begin{defi}
	Let $\mathcal{S}\subset\mathbb{C}\cup\{\infty\}$ be non-empty and $k\in\mathbb{N}\cup\{0\}\cup\{\infty\}$. We denote by $E_{f}(\mathcal{S},k)$ the set $E_{f}(\mathcal{S},k)=\displaystyle\bigcup_{a\in\mathcal{S}}E_{f}(a,k)$.\par Clearly $E_{f}(\mathcal{S})=E_{f}(\mathcal{S},\infty)$ and $\ol E(\mathcal{S},k)=E_{f}(\mathcal{S},0)$.   
\end{defi}\par
With the help of wighted sharing of sets, \emph{Banerjee - Mukherjee} \cite{Ban & Muk-HJ-2008} obtained the following results.
\begin{theoD}\cite{Ban & Muk-HJ-2008}
		Let $\mathcal{S}_1=\{0\}$, $\mathcal{S}_2=\{z:z^n+bz^{n-1}+c=0\}$ and $\mathcal{S}_3=\{\infty\}$, where $a$, $b$ are non-zero constants such that $z^n+bz^{n-1}+c=0$ has no repeated root and $n\geq 3$ is an integer. If for two non-constant meromorphic functions $f$ and $g$ having no simple pole satisfying, $E_{f}(\mathcal{S}_{1},1)=E_{g}(\mathcal{S}_{1},1)$, $E_{f}(\mathcal{S}_{2},5)=E_{g}(\mathcal{S}_{2},5)$ and $E_{f}(\mathcal{S}_{3},\infty)=E_{g}(\mathcal{S}_{3},\infty)$, then $f\equiv g$.  
\end{theoD}
\begin{theoE}\cite{Ban & Muk-HJ-2008}
		Let $\mathcal{S}_1=\{0\}$, $\mathcal{S}_2=\{z:z^n+bz^{n-1}+c=0\}$ and $\mathcal{S}_3=\{\infty\}$, where $a$, $b$ are non-zero constants such that $z^n+bz^{n-1}+c=0$ has no repeated root and $n\geq 3$ is an integer. If for two non-constant meromorphic functions $f$ and $g$ satisfying, $E_{f}(\mathcal{S}_{1},0)=E_{g}(\mathcal{S}_{1},0)$, $E_{f}(\mathcal{S}_{2},6)=E_{g}(\mathcal{S}_{2},6)$,  $E_{f}(\mathcal{S}_{3},\infty)=E_{g}(\mathcal{S}_{3},\infty)$ and $\delta_{1)}(\infty;f)+\delta_{1)}(\infty;g)>\displaystyle\frac{5}{n}$, then $f\equiv g$.
\end{theoE}
\begin{theoF}\cite{Ban & Muk-HJ-2008}
	Let $\mathcal{S}_1=\{0\}$, $\mathcal{S}_2=\{z:z^n+bz^{n-1}+c=0\}$ and $\mathcal{S}_3=\{\infty\}$, where $a$, $b$ are non-zero constants such that $z^n+bz^{n-1}+c=0$ has no repeated root and $n\geq 4$ is an integer. If for two non-constant meromorphic functions $f$ and $g$ satisfying, $E_{f}(\mathcal{S}_{1},0)=E_{g}(\mathcal{S}_{1},0)$, $E_{f}(\mathcal{S}_{2},6)=E_{g}(\mathcal{S}_{2},6)$,  $E_{f}(\mathcal{S}_{3},4)=E_{g}(\mathcal{S}_{3},4)$ and $\delta_{1)}(\infty;f)+\delta_{1)}(\infty;g)>0$, then $f\equiv g$.
\end{theoF}

\par Recently \emph{Banerjee - Majumder} \cite{Ban & Maj-A-2014} obtained two  results by improving some earlier results of \emph{Banerjee} \cite{Ban-APM-2007, Ban-KMJ-2009} as follows.
\begin{theoG}\cite{Ban & Maj-A-2014}
	Let $\mathcal{S}_1=\{0\}$, $\mathcal{S}_2=\{z:z^n+az^{n-1}+b=0\}$ and $\mathcal{S}_3=\{\infty\}$, where $a, b$ are non-zero constants such that $z^n+az^{n-1}+b=0$ has no repeated root and $n(\geq 4)$ be an integer. If for two non-constant meromorphic functions $f$ and $g$, $E_{f}(\mathcal{S}_1,k_1)=E_{g}(\mathcal{S}_1,k_1)$, $E_{f}(\mathcal{S}_2,k_2)=E_{g}(\mathcal{S}_2,k_2)$ and $E_{f}(\mathcal{S}_3,k_3)=E_{g}(\mathcal{S}_3,k_3)$, where $k_1\geq 0$, $k_2\geq 3$, $k_3\geq 1$ are integers satisfying \beas 3k_1k_2k_3>k_2+3k_1+k_3-2k_2k_3+4\;\;\text{and}\;\; \Theta_{f}+\Theta_{g}>0, \eeas where $\Theta_{h}=\Theta(\infty;h)+\Theta\left(\displaystyle\frac{a(1-n)}{n};h\right)$,  then $f\equiv g$. 
\end{theoG}
\begin{theoH}\cite{Ban & Maj-A-2014}
	Let $\mathcal{S}_1=\{0\}$, $\mathcal{S}_2=\{z:z^n+az^{n-1}+b=0\}$ and $\mathcal{S}_3=\{\infty\}$, where $a, b$ are non-zero constants such that $z^n+az^{n-1}+b=0$ has no repeated root and $n(\geq 3)$ be an integer. If for two non-constant meromorphic functions $f$ and $g$, $E_{f}(\mathcal{S}_1,k_1)=E_{g}(\mathcal{S}_1,k_1)$, $E_{f}(\mathcal{S}_2,k_2)=E_{g}(\mathcal{S}_2,k_2)$ and $E_{f}(\mathcal{S}_3,k_3)=E_{g}(\mathcal{S}_3,k_3)$, where $k_1\geq 0$, $k_2\geq 4$, $k_3\geq 1$ are integers satisfying \beas 2k_1k_2k_3>k_2+2k_1+k_3-k_2k_3+3\;\;\text{and}\;\; \Theta_{f}+\Theta_{g}>1, \eeas  where $\Theta_{h}=\Theta(\infty;h)+\Theta\left(\displaystyle\frac{a(1-n)}{n};h\right)$ then $f\equiv g$. 
\end{theoH}\par Earlier the problem of finding the possible answer of the \emph{\sc Question \ref{qn1.2}} was solved by \emph{Lin - Yi} \cite{Lin & Yi-TA-2003} who answered the \emph{\sc Question \ref{qn1.2}} by considering the sets $\mathcal{S}_1=\{0\}$, $\mathcal{S}_2=\{z:az^n-n(n-1)z^2+2n(n-2)bz=(n-1)(n-2)b^2\}$ and $\mathcal{S}_3=\{\infty\}$ for $n\geq 5$, where $a, b$ are constants such that $ab^{n-2}\neq 0, 2$.\par In \cite{Ban & Aha-BPAS-2014}, \emph{Banerjee - Ahamed}  modified the sets $\mathcal{S}_1$, $\mathcal{S}_2$ so that $\mathcal{S}_1=\{0,1\}$, and the number of elements in the new set $\mathcal{S}_2$ has decreased by $1$ in the optimal case. Moreover the conditions on the sharing sets $\mathcal{S}_{j}$, $(j=1,2,3)$ has also been relaxed to the conditions of sharing $(\mathcal{S}_{j},k_{j})$, $(j=1,2,3)$, where $(k_1,k_2,k_3)=(0,3,2), (0,4,1)$.\par From the above discussions, we have the following notes:
\begin{note}
	The lower bound of the cardinality of the main range set $ \mathcal{S}_2 $ is obtained so far in \emph{\sc Theorems A, B, D, E, H} and also in the result of \emph{Qiu -Fang} \cite{Qiu & Fan-BSMS-2002} is $3$ with the help of some extra suppositions.
\end{note}
\begin{note}
	Also one may check that the optimal choice for the weights $(k_1,k_2,k_3)=(0,3,1)$ in \emph{\sc Theorem G} can not be considered as it violates the condition $3k_1k_2k_3>k_2+3k_1+k_3-2k_2k_3+4$. 
\end{note}
\begin{note}
	We also see that in \emph{\sc Theorem H}, it is not possible to consider the weights as $(k_1,k_2,k_3)=(0,4,1)$ and hence as $(k_1,k_2,k_3)=(0,3,1)$.
\end{note}
    Based on the above observation, for the purpose of improving  all the above mentioned results further, one can ask the following question.
\begin{ques}\label{qn1.3}
	Can we obtain a uniqueness result corresponding to \emph{\sc Theorems A, B, D, E, H} and  \emph{Qiu -Fang} \cite{Qiu & Fan-BSMS-2002} without the help of any extra  suppositions in which the lower bound of the cardinality of the main range set will be $3$ ?
\end{ques}\par If the answer of the \emph{\sc Question \ref{qn1.3}} is found to be affirmative, then one natural question is as follows. 
\begin{ques}\label{qn1.4}
	Is it possible to reduce further the choice of the weights $(k_1,k_2,k_3)$ to $(0,3,1)$ in all the above mentioned results ? 
\end{ques}\par Answering \emph{\sc Questions \ref{qn1.2}}, \emph{\sc \ref{qn1.3}} and \emph{\sc\ref{qn1.4}} affirmatively is the main motivation of writing this paper. In this paper, we have modified the sets $\mathcal{S}_1=\{0\}$ by $\mathcal{S}_1=\{0,\delta_{a,b}^n\}$ and $\mathcal{S}_2$ by an new one and obtained two results out of which the second one directly improves all the above mentioned results.\par  To this end, we next suppose that $\displaystyle\delta_{a,b}^n=\frac{b(1-n)}{na},$ where $a, b\in\mathbb{C^{*}}$ and $n\geq 3$ be an integer. We consider here the $ \mathcal{S}_1=\{0,\delta_{a,b}^n\} $ as the set of zeros of the derivative of the polynomial $ az^n+bz^{n-1}+c. $

\begin{theo}\label{t1.2}
For $n\geq 3$, let $\mathcal{S}_1=\{0,\delta_{a,b}^n\}$, $\mathcal{S}_2=\{z:az^n+bz^{n-1}+c=0\}$ and $\mathcal{S}_3=\{\infty\}$, where $a, b, c\in\mathbb{C^{*}}=\mathbb{C}\setminus\{0\}$ be so chosen that $az^n+bz^{n-1}+c=0$ has no repeated root, $c\neq\displaystyle -\frac{b}{2n}\left(\delta_{a,b}^n\right)^{n-1}$. If for two non-constant meromorphic functions $f$ and $g$, $E_{f}(\mathcal{S}_1,0)=E_{g}(\mathcal{S}_1,0)$, $E_{f}(\mathcal{S}_2,n)=E_{f}(\mathcal{S}_2,n)$ and $E_{f}(\mathcal{S}_3,n-2)=E_{f}(\mathcal{S}_3,n-2)$, then $f\equiv g$.
\end{theo}  
\begin{cor}\label{c1.2}
	Let $\mathcal{S}_1=\bigg\{0,-\displaystyle\frac{2b}{3a}\bigg\}$, $\mathcal{S}_2=\{z:az^3+bz^2+c=0\}$ and $\mathcal{S}_3=\{\infty\}$, where $a, b, c\in\mathbb{C^{*}}$ be so chosen that $az^3+bz^2+c=0$ has no repeated root, $c\neq -\displaystyle\frac{2b^3}{27a^2}$. If for two non-constant meromorphic functions $f$ and $g$, $E_{f}(\mathcal{S}_1,0)=E_{g}(\mathcal{S}_1,0)$, $E_{f}(\mathcal{S}_2,3)=E_{f}(\mathcal{S}_2,3)$ and $E_{f}(\mathcal{S}_3,1)=E_{f}(\mathcal{S}_3,1)$, then $f\equiv g$.
\end{cor}
\begin{rem}
	Clearly \emph{\sc Corollary \ref{c1.2}}\; directly improves the above mentioned results as we see that the lower bound of $ n $ is $ 3 $, with the corresponding weights $( k_1, k_2, k_3 )=(0, 3, 1)$.
\end{rem}\par The following example shows that the conclusions of the   \emph{\sc Theorems \ref{t1.2}}\; cease to be hold if we consider $c=\displaystyle -\frac{b}{2n}\left(\delta_{a,b}^n\right)^{n-1}$.
\begin{exm}
	Let $a=1, b=3$, then $\displaystyle -\frac{b}{2n}\left(\delta_{a,b}^n\right)^{n-1}=-2$. Let $c=\displaystyle -\frac{b}{2n}\left(\delta_{a,b}^n\right)^{n-1}=-2$ and $\mathcal{S}_2=\{z:z^3+3z^2-2=0\}=\{-1, -1-\sqrt{3},-1-\sqrt{3}\}$ and $\mathcal{S}_3=\{\infty\}$. Hence we must have $\mathcal{S}_1=\{0,\delta_{a,b}^n\}=\{0,-2\}$. Let $f(z)=\phi(z)-2$ and $g(z)=-\phi(z)$, where $\phi(z)$ is a non-constant meromorphic function. It is clear that $E_{f}(\mathcal{S}_{j})=E_{g}(\mathcal{S}_{j})$ for $j=1,2,3$ and hence $E_{f}(\mathcal{S}_{1},0)=E_{g}(\mathcal{S}_{1},0)$, $E_{f}(\mathcal{S}_{2},3)=E_{g}(\mathcal{S}_{2},3)$ and $E_{f}(\mathcal{S}_{3},1)=E_{g}(\mathcal{S}_{3},1)$ but note that $f\not\equiv g$.  
\end{exm}
\par The next example shows the sharpness of the cardinalities of the set $\mathcal{S}_1$ and the main range set $\mathcal{S}_2$ in the  \emph{\sc Theorem \ref{t1.2}}.
\begin{exm}
	Let $\mathcal{S}_2=\{z:az^2+bz+c=0\}=\{\gamma,\delta\}$, where $\gamma+\delta=-\displaystyle\frac{b}{a}$, $\gamma\delta=\displaystyle\frac{c}{a}$, $a,b,c\in\mathbb{C^{*}}$, $c\neq\displaystyle\frac{b^2}{8a}$. Hence $\mathcal{S}_1=\bigg\{-\displaystyle\frac{b}{2a}\bigg\}=\bigg\{\displaystyle\frac{\gamma+\delta}{2}\bigg\}$. Let $\mathcal{S}_3=\{\infty\}$ and $f(z)=h(z)+\gamma+\delta$ and $g(z)=-h(z)$, where $h(z)$ is any non-constant meromorphic function. We see that all the conditions of  \emph{\sc Theorem \ref{t1.2}} are satisfied but $f\not\equiv g$.
\end{exm} 
\par The following example shows that,  the condition $b\neq 0$, in \emph{\sc Theorem \ref{t1.2}},\; can not be removed.
\begin{exm}
	Let $b=0$, then $\delta_{a,b}^n=0$. Thus, we get $\mathcal{S}_1=\{0\}$. \beas \text{Let}\;\; \mathcal{S}_2=\{z:az^3+c=0\}=\bigg\{\sqrt[3]{-\displaystyle\frac{c}{a}}, \sqrt[3]{-\displaystyle\frac{c}{a}}\omega, \sqrt[3]{-\displaystyle\frac{c}{a}}\omega^2\bigg\},\eeas $a,c\in\mathbb{C^{*}}$, where $ \omega $ is a cube roots of unity and $\mathcal{S}_3=\{\infty\}$. Let $f(z)$ be a non-constant meromorphic function and $g(z)=\omega\; f(z)$, where $\omega$ is a non-real cube root of unity. It is clear that   $E_{f}(\mathcal{S}_{1},0)=E_{g}(\mathcal{S}_{1},0)$, $E_{f}(\mathcal{S}_{2},3)=E_{g}(\mathcal{S}_{2},3)$ and $E_{f}(\mathcal{S}_{3},1)=E_{g}(\mathcal{S}_{3},1)$  but $f\not\equiv g$.
\end{exm}
\par The next two examples show that the set $\mathcal{S}_2$ considered in \emph{\sc Theorem \ref{t1.2}} can not be replaced by any arbitrary set.
\begin{exm}
	Let $\mathcal{S}_1=\bigg\{\displaystyle\frac{6-\sqrt{3}}{3}, \frac{6+\sqrt{3}}{3}\bigg\}$, \beas \mathcal{S}_2=\bigg\{z:z^3-6z^2+11z-6=0\bigg\}=\{1,2,3\}\eeas and $\mathcal{S}_3=\{\infty\}$. Let $f(z)=h(z)+4$ and $g(z)=-h(z)$, where $ h(z) $ is a non-constant meromorphic function. Although we se that $E_{f}(\mathcal{S}_{1},0)=E_{g}(\mathcal{S}_{1},0)$, $E_{f}(\mathcal{S}_{2},3)=E_{g}(\mathcal{S}_{2},3)$ and $E_{f}(\mathcal{S}_{3},1)=E_{g}(\mathcal{S}_{3},1)$  but $f\not\equiv g$.
\end{exm}
\begin{exm}
	Let $\mathcal{S}_1=\bigg\{\displaystyle\frac{15-\sqrt{3}}{3}, \frac{15+\sqrt{3}}{3}\bigg\}$, \beas \mathcal{S}_2=\bigg\{z:z^3-15z^2+74z-120=0\bigg\}=\{4,5,6\}\eeas and $\mathcal{S}_3=\{\infty\}$. Let $f(z)=\phi(z)+10$ and $g(z)=-\phi(z)$, where $ \phi(z) $ is a non-constant meromorphic function. Although we se that $E_{f}(\mathcal{S}_{1},0)=E_{g}(\mathcal{S}_{1},0)$, $E_{f}(\mathcal{S}_{2},3)=E_{g}(\mathcal{S}_{2},3)$ and $E_{f}(\mathcal{S}_{3},1)=E_{g}(\mathcal{S}_{3},1)$  but $f\not\equiv g$.
\end{exm}
\begin{note}
	One can find many examples by considering $ \mathcal{S}_1 $ as th set of roots of the derivative of the polynomial of degree $ 3 $ whose roots formed the set $ \mathcal{S}_2 $, where $ \mathcal{S}_2=\{m, m+1, m+2\} $, where $ m\in\mathbb{N} $, and by choosing the functions $ f(z)=h(z)+2(m+1) $ and $ g(z)=-h(z) $, where $ h(z) $ is a non-constant meromorphic function. 
\end{note}
%--------------------------------------------------------------------------------------------------------------------%
\section{\sc Some lemmas}
In this section, we are going to discuss some lemmas which will be needed later to prove our main results. We define, for two non-constant meromorphic functions $f$ and $g$, \bea\label{e2.1}   \mathcal{F}=\frac{f^{n-1}(af+b)}{-c},\;\;\;\mathcal{G}=\frac{g^{n-1}(ag+b)}{-c}.  \eea\par Associated to $\mathcal{F}$ and $\mathcal{G}$, we next define $\mathcal{H}$ as follows: \bea\label{e2.2} \mathcal{H}=\left(\frac{\mathcal{F}^{\prime\prime}}{\mathcal{F}^{\prime}}-\frac{2\mathcal{F}^{\prime}}{\mathcal{F}-1}\right)-\left(\frac{\mathcal{G}^{\prime\prime}}{\mathcal{G}^{\prime}}-\frac{2\mathcal{G}^{\prime}}{\mathcal{G}-1}\right) \eea and \bea\label{e2.4}   \Psi=\frac{\mathcal{F}^{\prime}}{\mathcal{F}-1}-\frac{\mathcal{G}^{\prime}}{\mathcal{G}-1}.\eea 
\begin{lem}\label{lem2.1}\cite{Mok-1971} Let $ h $ be a non-constant meromorphic function and let \beas \mathcal{R}(h)=\displaystyle\frac{\displaystyle\sum_{i=1}^{n}a_ih^i}{\displaystyle\sum_{j=1}^{m}b_jh^j}, \eeas be an irreducible rational function in $ g $ with constant coefficients $\{a_i\} $, $ \{b_j\}$, where $ a_n\neq 0 $ and $ b_m\neq 0 $. Then \beas T(r,\mathcal{R}(h))=\max\{n,m\}\; T(r,h)+S(r,h). \eeas
\end{lem}
\begin{lem}\label{lem2.2}
	Let  $\mathcal{F}$ and $\mathcal{G}$ be given by (\ref{e2.1}) satisfying  $E_{_{\mathcal{F}}}(1,q)=E_{_{\mathcal{G}}}(1,q)$, $0\leq q<\infty$ with $\mathcal{H}\not\equiv 0$, then \beas N_E^{1)}\left(r,\frac{1}{\mathcal{F}-1}\right)=N_E^{1)}\left(r,\frac{1}{\mathcal{G}-1}\right)&\leq& N(r,\mathcal{H})+S(r,\mathcal{F})+S(r,\mathcal{G}).\eeas 
\end{lem}
\begin{proof} Since $E_{_{\mathcal{F}}}(1,q)=E_{_{\mathcal{G}}}(1,q)$. It is clear that any simple $1$-point of $\mathcal{F}$ and $\mathcal{G}$ is a zero of $\mathcal{H}$. From the construction of $\mathcal{H}$, we know that $m(r,\mathcal{H})=S(r,\mathcal{F})+S(r,\mathcal{G}).$ Therefore by \emph{First Fundamental Theorem}, we get \beas && N_E^{1)}\left(r,\frac{1}{\mathcal{F}-1}\right)=N_E^{1)}\left(r,\frac{1}{\mathcal{G}-1}\right)\\ &\leq& N\left(r,\frac{1}{\mathcal{H}}\right) \\&\leq& N(r,\mathcal{H})+S(r,\mathcal{F})+S(r,\mathcal{G}). \eeas	
\end{proof}
\begin{lem}\label{lem2.3}
	Let the set $\mathcal{S}_2$ be given as in \emph{Theorem \ref{t1.2}}\; and $\Psi$ is given by \emph{(\ref{e2.4})}. If $E_{f}(\mathcal{S}_2,n)=E_{g}(\mathcal{S}_2,n)$ and $E_{f}(\mathcal{S}_3,n-2)=E_{g}(\mathcal{S}_3,n-2)$ and $\Psi\not\equiv 0$, then \beas && \ol N\left(r,\frac{1}{f}\right)+ \ol N\left(r,\frac{1}{f-\delta_{a,b}^n}\right)\\&\leq& \ol N\left(r,\frac{1}{\mathcal{F}-1}\mid \geq n+1\right)+\ol N(r,f\mid\geq n-1)+S(r,f).\eeas 
\end{lem}
\begin{proof}
Since $\Psi\not\equiv 0$, so in view of lemma of logarithmic derivatives, we have $m(r,\Psi)=S(r,f)$. Again since $E_{f}(\mathcal{S}_2,n)=E_{g}(\mathcal{S}_2,n)$ and $E_{f}(\mathcal{S}_3,n-2)=E_{g}(\mathcal{S}_3,n-2)$, then one can note that \bea\label{e2.9} N(r,\Psi)\leq \ol N\left(r,\frac{1}{\mathcal{F}-1}\mid\geq n+1\right)+\ol N(r,f\mid\geq n-1)+S(r,f). \eea\par Let $z_0$ be a point such that $f(z_0)=0$ or $f(z_0)=\delta_{a,b}^n$. Then since $E_{f}(\mathcal{S}_1,0)=E_{g}(\mathcal{S}_1,0)$, so we must have $\Psi(z_0)=0$. Thus we see that \bea\label{e2.10} \ol N\left(r,\frac{1}{f}\right)+\ol N\left(r,\frac{1}{f-\delta_{a,b}^n}\right)\leq N\left(r,\frac{1}{\Psi}\right).\eea\par Applying the \emph{First Fundamental Theorem}, we get from (\ref{e2.9}) and (\ref{e2.10}), \beas && \ol N\left(r,\frac{1}{f}\right)+\ol N\left(r,\frac{1}{f-\delta_{a,b}^n}\right)\\ &\leq& N\left(r,\frac{1}{\Psi}\right)\\ &\leq& T(r,\Psi)+S(r,f)\\&=& N(r,\Psi)+m(r,\Psi)+S(r,f)\\&=&N(r,\Psi)+S(r,f)\\ &\leq& \ol N\left(r,\frac{1}{\mathcal{F}-1}\mid\geq n+1\right)+\ol N(r,f\mid\geq n-1)+S(r,f).  \eeas
\end{proof}
\begin{lem}\label{lem2.4}\cite{Fang-1999}
	Let $a_1$, $a_2$, $a_3$, $a_4$ be four distinct complex numbers. If $E_{f}(a_j,\infty)=E_{g}(a_j,\infty)$, (j=1, 2, 3, 4), then $f(z)=\displaystyle\frac{\alpha\;g(z)+\beta}{\gamma\; g(z)+\delta}$, where $\alpha\delta-\beta\gamma\neq 0$.
\end{lem}
\begin{lem}\label{lem2.5}\cite{Fang-1999}
	If $E_{f^{*}}(1,\infty)=E_{g^{*}}(1,\infty)$ with $\delta_{2}(0;f^{*})+\delta_{2}(0;g^{*})+\delta_{2}(\infty,f^{*})+\delta_{2}(\infty,g^{*})>3$, then either $f^{*}g^{*}\equiv 1$ or $f^{*}\equiv g^{*}.$
\end{lem}		
%-------------------------------------------------------------------------------------------------------------------------%
\section{\sc Proof of the theorem}

\begin{proof}[Proof of Theorem \ref{t1.2}] Let $\mathcal{F}$ and $\mathcal{G}$ be given by (\ref{e2.1}) and $\mathcal{H}$, by (\ref{e2.2}).  We now discuss the following cases.\\
	\noindent{\sc Case 1.} Let if possible $\mathcal{H}\not\equiv 0$. Therefore it is clear that $\mathcal{F}\not\equiv\mathcal{G}$ and hence $\Psi\not\equiv 0$. By the lemma of logarithmic derivatives, one can easily get that $m(r,\mathcal{H})=S(r,f)+S(r,g)=m(r,\Psi)$. Since $E_{f}(\mathcal{S}_1,0)=E_{g}(\mathcal{S}_1,0)$, $E_{f}(\mathcal{S}_2,n)=E_{g}(\mathcal{S}_2,n)$ and $E_{f}(\mathcal{S}_3,n-2)=E_{g}(\mathcal{S}_3,n-2)$ hence from the construction of $\mathcal{H}$, one can easily get that  \bea\label{e3.1} && N(r,\mathcal{H})\\&\leq& \nonumber\ol N\left(r,\frac{1}{\mathcal{F}-1}\mid\geq n+1\right)+\ol N(r,f\mid\geq n-1)+\ol N\left(r,\frac{1}{f}\right)+\ol N\left(r,\frac{1}{f-\delta_{a,b}^n}\right)\\ &&\nonumber+ N_{0}\left(r,\frac{1}{f^{\prime}}\right)+ N_{0}\left(r,\frac{1}{g^{\prime}}\right)+S(r,f)+S(r,g)\nonumber,\eea\par where $N_{0}\left(r,\displaystyle\frac{1}{f^{\prime}}\right)$ denote the counting function of those zeros of $f^{\prime}$ which are not the zeros of $f(f-\delta_{a,b}^n)(\mathcal{F}-1)$. Similarly, $N_{0}\left(r,\displaystyle\frac{1}{g^{\prime}}\right)$ can be defined.\par 
	By applying \emph{Second Fundamental Theorem}, we get \bea\label{e3.2} && (n+1)\bigg\{T(r,f)+T(r,g)\bigg\}\\ &\leq&\nonumber \ol N\left(r,\frac{1}{\mathcal{F}-1}\right)+\ol N(r,f)+\ol N\left(r,\frac{1}{f}\right)+\ol N\left(r,\frac{1}{f-\delta_{a,b}^n}\right)+\ol N\left(r,\frac{1}{\mathcal{G}-1}\right)\\ && \nonumber+\ol N(r,g)+\ol N\left(r,\frac{1}{g}\right)+\ol N\left(r,\frac{1}{g-\delta_{a,b}^n}\right)-N_{0}\left(r,\frac{1}{f^{\prime}}\right)-N_{0}\left(r,\frac{1}{g^{\prime}}\right)\\ &&+S(r,f)+S(r,g).\nonumber \eea \par 
	Now by using \emph{\sc Lemmas \ref{lem2.2}, \ref{lem2.3}}\; and (\ref{e3.1}), we get from (\ref{e3.2}) 
	\bea\label{e3.3} && (n+1)\bigg\{T(r,f)+T(r,g)\bigg\}\\ &\leq&\nonumber \ol N\left(r,\frac{1}{\mathcal{F}-1}\mid\geq n+1\right)+N_{n-1}(r,f)+\ol N\left(r,\frac{1}{\mathcal{G}-1}\right)+\ol N\left(r,\frac{1}{\mathcal{F}-1}\mid\geq 2\right)\\ &&\nonumber \ol N(r,g)+3\bigg\{\ol N(r,f\mid\geq n-1)+\ol N\left(r,\frac{1}{\mathcal{F}-1}\mid\geq n+1\right)\bigg\}+S(r,f)+S(r,g)\\ &\leq&\nonumber N_{n-1}(r,f)+N_{n-1}(r,g)+\frac{1}{n-1} N(r,f)+\frac{1}{n-1} N(r,g)\\ &&\nonumber+\bigg\{2\ol N\left(r,\frac{1}{\mathcal{F}-1}\mid\geq n+1\right)+\ol N\left(r,\frac{1}{\mathcal{F}-1}\mid\geq 2\right)\bigg\}\\ &&\nonumber+\bigg\{\ol N\left(r,\frac{1}{\mathcal{G}-1}\right)+2\ol N\left(r,\frac{1}{\mathcal{F}-1}\mid\geq n+1\right)\bigg\}+S(r,f)+S(r,g). \eea\par Next, we see that \bea\label{e3.4} && \frac{1}{2}\ol N\left(r,\frac{1}{\mathcal{F}}\mid\leq 1\right)+\ol N\left(r,\frac{1}{\mathcal{F}}\mid\geq 2\right)+2\ol  N\left(r,\frac{1}{\mathcal{F}}\mid\geq n+1\right)\\ &\leq&\nonumber \frac{1}{2} N\left(r,\frac{1}{\mathcal{F}-1}\right)+\ol N\left(r,\frac{1}{\mathcal{F}-1}\geq n+1\right)\\ &\leq&\nonumber \left(\frac{1}{2}+\frac{1}{n+1}\right) N\left(r,\frac{1}{\mathcal{F}-1}\right)=\frac{n+3}{2(n+1)}N\left(r,\frac{1}{\mathcal{F}-1}\right). \eea\par Similarly, we get \bea\label{e3.5} && \frac{1}{2}\ol N\left(r,\frac{1}{\mathcal{G}}\mid\leq 1\right)+\ol N\left(r,\frac{1}{\mathcal{G}}\mid\geq 2\right)+2 \ol N\left(r,\frac{1}{\mathcal{G}}\mid\geq n+1\right)\\ &\leq&\nonumber \frac{n+3}{2(n+1)}N\left(r,\frac{1}{\mathcal{G}-1}\right).\eea\par Therefore, using (\ref{e3.4}) and (\ref{e3.5}), we obtain from (\ref{e3.3}) \beas && (n+1)\bigg\{T(r,f)+T(r,g)\bigg\}\\ &\leq& \left(1+\frac{1}{n-1}+\frac{n(n+3)}{2(n+1)}\right)\bigg\{T(r,f)+T(r,g)\bigg\},\eeas which contradicts $n\geq 3$.\\
	 \noindent{\sc Case 2.} Therefore $\mathcal{H}\equiv 0$.\par Then on integrating, we get from (\ref{e2.2}) \bea\label{e3.6} \frac{1}{\mathcal{F}-1}\equiv\frac{\mathcal{A}}{\mathcal{G}-1}+\mathcal{B},\;\;\text{where}\;\; \mathcal{A}(\neq 0),\mathcal{B}\in\mathbb{C}.\eea\par From (\ref{e3.6}), we obtain in view Lemma \ref{lem2.1} that \beas T(r,f)=T(r,g)+S(r,f)+S(r,g).\eeas \par Let $ \infty $ is a e.v.P of $ f $. Then we must have $ \ol N(r,f)=S(r,f) $.\par From the proof of Lemma \ref{lem2.3}, we have already \bea\label{e33.77} && \ol N\left(r,\frac{1}{f}\right)+\ol N\left(r,\frac{1}{f-\delta_{a,b}^n}\right)\\ &\leq& \nonumber \ol N\left(r,\frac{1}{\mathcal{F}-1}\mid\geq n+1\right)+\ol N(r,f\mid\geq n-1)+S(r,f)\\&\leq&\nonumber \frac{1}{n+1} N\left(r,\frac{1}{\mathcal{F}-1}\right)+\frac{1}{n-1}\ol N(r,f)+S(r,f)\\&\leq&\frac{n}{n+1} T(r,f)+S(r,f)\nonumber. \eea\par By the \emph{Second Fundamental Theorem} and (\ref{e33.77}), we obtain \beas T(r,f)&\leq& \ol N\left(r,\frac{1}{f}\right)+\ol N\left(r,\frac{1}{f-\delta^n_{a,b}}\right)+\ol N(r,f)+S(r,f)\\&\leq& \frac{n}{n+1} T(r,f)+S(r,f),\eeas which is a contradiction.\par 
	 	 
	 Let $\infty$ is not a \textit{e.v.P} of $f$. So there must exits $z_0\in\mathbb{C}$ such that $f(z_0)=\infty$. Since $E_{f}(\mathcal{S}_3,n-2)=E_{g}(\mathcal{S}_3,n-2)$, so get from (\ref{e3.6}) that $\mathcal{B}=0$.\par Therefore, we have \beas  \mathcal{A}(\mathcal{F}-1)\equiv (\mathcal{G}-1). \eeas\par i.e., we have \bea\label{ee3.2} \mathcal{A}(af^n+bf^{n-1}+c)\equiv (ag^n+bg^{n-1}+c).  \eea\par 
Since $E_{f}(\mathcal{S}_1,0)=E_{g}(\mathcal{S}_1,0)$, then we have the following two possibilities. \begin{enumerate} 
	 	\item[(i)] $ E_{f}(0,0)=E_{g}(0,0)\;\;\text{and}\;\; E_{f}(\delta_{a,b}^n,0)=E_{g}(\delta_{a,b}^n,0),$ or
	 	\item[(ii)] $ E_{f}(0,0)=E_{g}(\delta_{a,b}^n,0)\;\;\text{and}\;\; E_{f}(\delta_{a,b}^n,0)=E_{g}(0,0).$
	 \end{enumerate}
	 \noindent{\sc Subcase 2.1.} Suppose $ E_{f}(0,0)=E_{g}(0,0)\;\;\text{and}\;\; E_{f}(\delta_{a,b}^n,0)=E_{g}(\delta_{a,b}^n,0).$ Then there exist $z_0, z_1\in\mathbb{C}$ such that $f(z_0)=0=g(z_0)$ and $f(z_1)=\delta_{a,b}^n=g(z_1)$. In both the cases, we get from (\ref{ee3.2}) that $\mathcal{A}=1$. Then (\ref{ee3.2}) reduces to \bea\label{ee3.3} af^n+bf^{n-1}\equiv ag^n+bg^{n-1}.\;\;\text{i.e.,}\;\; f^{n-1}(af+b)\equiv g^{n-1}(ag+b)\eea\par Since $E_{f}(0,0)=E_{g}(0,0)$, so from (\ref{ee3.3}), we get $E_{f}\left(-\displaystyle\frac{b}{a},0\right)=E_{g}\left(-\displaystyle\frac{b}{a},0\right)$. Again since $E_{f}(\mathcal{S}_3,n-2)=E_{g}(\mathcal{S}_3,n-2)$, thus we see that \beas E_{f}(0,0)=E_{g}(0,0),\;\;\; E_{f}\left(\displaystyle\delta_{a,b}^n,0\right)=E_{g}\left(\displaystyle\delta_{a,b}^n,0\right),\\ E_{f}\left(-\frac{b}{a},0\right)=E_{g}\left(-\frac{b}{a},0\right),\;\;\;\; E_{f}(\infty,n-2)=E_{g}(\infty,n-2).\eeas\par Then by \emph{\sc Lemma \ref{lem2.4}}, one must have \bea\label{ee3.4} f(z)=\displaystyle\frac{\alpha\;g(z)+\beta}{\gamma\; g(z)+\delta},\eea\par where $\alpha\delta-\beta\gamma\neq 0$.\par Therefore, equations (\ref{ee3.3}) and (\ref{ee3.4}) combinedly give $f\equiv g$.\\ 
	 \noindent{\sc Subcase 2.2.} Suppose $ E_{f}(0,0)=E_{g}(\delta_{a,b}^n,0)\;\;\text{and}\;\; E_{f}(\delta_{a,b}^n,0)=E_{g}(0,0).$\par We now discuss the following subcases.\\
	 \noindent{\sc Subcase 2.2.1.} Let both $E_{f}(0,0)=E_{g}(\delta_{a,b}^n,0)=\phi$ and $E_{f}(\delta_{a,b}^n,0)=E_{g}(0,0)=\phi$. Since $E_{f}(\infty,n-2)=E_{g}(\infty,n-2)$, so we must have $E_{f^{*}}(1,n-2)=E_{g^{*}}(1,n-2)$, where $f^{*}(z)=\displaystyle\frac{f(z)}{f(z)-\delta_{a,b}^n}\neq 0, \infty$ and $g^{*}(z)=\displaystyle\frac{g(z)-\delta_{a,b}^n}{g(z)}\neq 0, \infty$. Again we note that \beas\delta_{2}(0;f^{*})+\delta_{2}(0;g^{*})+\delta_{2}(\infty,f^{*})+\delta_{2}(\infty,g^{*})=4>3.\eeas\par Therefore, by using \emph{\sc Lemma  \ref{lem2.5}}, we have $f^{*}\equiv g^{*}$ or $f^{*}g^{*}\equiv 1$.\\
	 \noindent{\sc Subcase 2.2.1.1.} Suppose $f^{*}g^{*}\equiv 1$. Then we have $f\equiv g$.\\
	 \noindent{\sc Subcase 2.2.1.2.} Suppose $f^{*}\equiv g^{*}$. Then we have \bea\label{ee3.5} f+g=\delta_{a,b}^n.\eea\par Thus from (\ref{ee3.2}) and (\ref{ee3.5}), we see that $f$ is a constant, which is absurd.\\
	 \noindent{\sc Subcase 2.2.2.} Let $E_{f}(0,0)=E_{g}(\delta_{a,b}^n,0)=\phi$ or $E_{f}(\delta_{a,b}^n,0)=E_{g}(0,0)=\phi$.\\
	 \noindent{\sc Subcase 2.2.2.1.} Suppose $E_{f}(0,0)=E_{g}(\delta_{a,b}^n,0)=\phi$ and $E_{f}(\delta_{a,b}^n,0)=E_{g}(0,0)\neq\phi$. This implies that there exists $z_0\in\mathbb{C}$, such that $f(z_0)=\delta_{a,b}^n$ and $g(z_0)=0$. So  from (\ref{ee3.2}), we get \bea\label{ee3.6}  \mathcal{A}=\frac{a\left(\delta_{a,b}^n\right)^n+b\left(\delta_{a,b}^n\right)^{n-1}+c}{c}.\eea\par 
	 It follows from (\ref{ee3.6}) that \bea\label{e3.13}- a\left(\delta_{a,b}^n\right)^n-b\left(\delta_{a,b}^n\right)=c\left(1-\mathcal{A}\right). \eea Clearly, one root of the equation (\ref{e3.13}) is $ \delta^n_{a,b} $ of multiplicity $ 2 $. Equation (\ref{ee3.2}) can be written as \bea\label{e3.14} af^n+bf^{n-1}+c-\frac{c}{\mathcal{A}}=\frac{1}{\mathcal{A}}\left(ag^n+bg^{n-1}\right).\eea We must have $ c-\frac{c}{\mathcal{A}}\neq c\left(1-\mathcal{A}\right), $ otherwise we will have $ \mathcal{A}=\pm 1 $, which is a contradiction as $c\neq\displaystyle -\frac{b}{2n}\left(\delta_{a,b}^n\right)^{n-1}$, $ \delta^n_{a,b}\neq -\frac{b}{a},\; 0 $.\par Now, equation (\ref{e3.14}) can be written as \bea\label{e3.15} a\prod_{j=1}^{n}\left(f-\zeta_j\right)=\frac{1}{\mathcal{A}}g^{n-1}(ag+b), \eea where $ \zeta_j\; (j=1, 2, \ldots, n) $ are distinct roots of the polynomial $\displaystyle af^n+bf^{n-1}+c-\frac{c}{\mathcal{A}}. $ From (\ref{e3.15}), it is clear that $ 0 $ is e.v.P of $ g $, which contradicts our assumption  $E_{f}(\delta_{a,b}^n,0)=E_{g}(0,0)\neq\phi$.
	 \noindent{\sc Subcase 2.2.2.2.} Suppose $E_{f}(0,0)=E_{g}(\delta_{a,b}^n,0)\neq\phi$ and $E_{f}(\delta_{a,b}^n,0)=E_{g}(0,0)=\phi$. This implies that there exists $z_1\in\mathbb{C}$, such that $f(z_1)=0$ and $g(z_1)=\delta_{a,b}^n$. Then from (\ref{ee3.2}), we get \bea\label{ee3.9}  \mathcal{A}=\frac{c}{a\left(\delta_{a,b}^n\right)^n+b\left(\delta_{a,b}^n\right)^{n-1}+c}.\eea\par Next proceeding exactly same way as done in the \emph{\sc Subcase 2.2.2.1}, we get a contradiction.\\
	 \noindent{\sc Subcase 2.2.3.} Suppose both $E_{f}(0,0)=E_{g}(\delta_{a,b}^n,0)\neq\phi$ and $E_{f}(\delta_{a,b}^n,0)=E_{g}(0,0)\neq\phi$. Then we get \beas  \mathcal{A}=\frac{a\left(\delta_{a,b}^n\right)^n+b\left(\delta_{a,b}^n\right)^{n-1}+c}{c}\;\;\text{and}\;\;  \mathcal{A}=\frac{c}{a\left(\delta_{a,b}^n\right)^n+b\left(\delta_{a,b}^n\right)^{n-1}+c}.\eeas\par Thus we see that $\mathcal{A}=\pm 1$, which contradicts $c\neq\displaystyle -\frac{b}{2n}\left(\delta_{a,b}^n\right)^{n-1}.$\par This completes the proof.
\end{proof}

%______________________________________________________%
\section{\sc Concluding remarks and a question}
\par
In this paper, we proved a result with the best possible cardinalities of the three sets sharing problems till now by answering the question posed by \emph{Yi} \cite{Yi-SC-1994} without the help of any extra suppositions. We have also abled to relax the nature of sharing of the sets compare to other results mentioned in the introduction. But we don't know whether the choice of the weights $(k_1,k_2,k_3)=(0,3,1)$ associated with the corresponding sets $ \mathcal{S}_j $, $ j=1, 2, 3 $, in our main result is the best possible or not. So we have the following quarry for the future investigation in this direction.
\begin{ques}
	Keeping all other conditions intact in \emph{Theorem \ref{t1.2}}, is it possible to relax the nature of sharing of the sets further ?
\end{ques}
\noindent{\bf Acknowledgment} The author would like to thank the referee for his/her helpful suggestions
and comments towards the improvement of this manuscript.

\end{document}